\documentclass{amsart}

\usepackage{amssymb}
\usepackage{overpic}
\usepackage{graphics}
\usepackage{epsfig}
\usepackage{amsmath}
\usepackage{amsfonts}
\usepackage{paralist}

\newtheorem{theorem}{Theorem}[section]
\newtheorem{corollary}{Corollary}
\newtheorem{example}{Example}

\newtheorem{lemma}[theorem]{Lemma}

\theoremstyle{definition}
\newtheorem{definition}[theorem]{Definition}

\begin{document}
\title[Improved Estimates of Survival Prob. via Isospectral Transformations]{Improved Estimates of Survival Probabilities via Isospectral Transformations}

\author{L. A. Bunimovich$^1$}

\author{B. Z. Webb$^2$}



\maketitle


\begin{center}
$^{1}$ \smaller{ABC Math Program and School of Mathematics, Georgia Institute of Technology, 686 Cherry Street, Atlanta, GA 30332, USA}\\
$^{2}$ Rockefeller University, Laboratory of Statistical Physics, 1230 York Avenue, New York, NY 10065, USA\\
E-mail: bunimovih@math.gatech.edu and bwebb@rockefeller.edu
\end{center}

\begin{abstract}
We consider open systems generated from one-dimensional maps that admit a finite Markov partition and use the recently developed theory of isospectral graph transformations to estimate a system's survival probabilities. We show that these estimates are better than those obtained through a more direct approach.
\end{abstract}

\section{Introduction}
Recently, a nontrivial relation between the dynamics of networks and the dynamics of open systems has been found. This discovery has lead to advances in the analysis of network dynamics and has also introduced a new research direction concerned with the finite time properties of open systems with finite sized holes.

The main idea behind this approach is that network dynamics can be broken down into three parts; (i) the network's graph structure, often referred to as its \emph{topology}, (ii) the local or intrinsic dynamics of the network elements, and (iii) the network interactions between these elements. To each of (i)-(iii) there is an associated dynamical system, which together can be used to characterize the dynamics of the network \cite{AB07,BB11,B12,BY10}. The same approach can be used to study the topological properties of open dynamical systems \cite{AB10,DW12}.

One of the fundamental concerns in the study of networks is understanding the relation between a network's structure and its dynamics. However, the networks we often encounter in either nature or engineering are typically very large, i.e. have a large number of elements. It is therefore tempting to want to reduce such networks by excluding some subset of these elements while preserving some important characteristic(s) of the original network.

Since real networks are typically dynamic, one might first consider the spectrum of the network's weighted adjacency matrix to be one such a characteristic worth preserving. That is, one could hope to find a way of reducing a network while maintaining its spectrum. However, it seems immediately clear that such an ``isospectral reduction" is impossible based on the Fundamental Theorem of Algebra.

In fact, it is possible to reduce the size of a network (or matrix) while preserving its spectrum. This theory of isospectral reductions can be found in \cite{BW10} and is part of a larger theory of isospectral transformations introduced in the paper. Such transformations were used in \cite{BW09} to improve each of the classical eigenvalue estimates associated with Gershgorin, et. al. \cite{Varga09}.

In the present paper we make use of this interplay between dynamical networks and open dynamical systems. We obtain estimates of survival probabilities for a class of one dimensional maps which admit a finite Markov partition. Most importantly, we shown that any isospectral transformation corresponding to an open system leads to sharper estimates of system's survival probabilities.

\section{Open and Closed Dynamical Systems}
Let $f:I\rightarrow I$ where $I=[0,1]$. For $0=q_0<q_1<\dots<q_{m-1}<q_m=1$ we let $\xi_i=(q_{i-1},q_i]$ for $1\leq i\leq m$ and assume that the following hold. First, the function $f|_{\xi_i}$ is differentiable for each $1\leq i\leq m$. Second, the sets $\xi_i=(q_{i-1},q_i]$ form a \emph{Markov partition} $\xi=\{\xi_i\}_{i=1}^m$ of $f$. That is, for each $1\leq i\leq m$ the closure $cl(f(\xi_i))$ is the interval $[q_{j},q_{j+k}]$ for some $k\geq 1$ and $j$ that depends on $i$.

We consider the situation where orbits of $f:I\rightarrow I$ escape through an element of the Markov partition $\xi=\{\xi_i\}_{i=1}^m$ or, more generally, some union $H$ of these partition elements. Equivalently, we can modify the function $f$ so that orbits cannot leave the set $H$ once they have entered it. Here, orbits that enter $H$ are considered to have escaped from the system. This later approach of modifying $f|_H$ turns out to be more convenient for our discussion and will be the direction we take. In what follows we let $M=\{1,\dots m\}$.

\begin{definition}
Let $H=\bigcup_{i\in \mathcal{I}}{\xi_i}$ where $\mathcal{I}\subset M$. We introduce the new map $f_H:I\rightarrow I$ defined by
$$f_H(x)=
\begin{cases}
f(x) \ \ &\text{if} \ \ x\notin H\\
x \ \ &\text{otherwise}
\end{cases}.$$
We call the set $H$ a \emph{hole} and the function $f_H:I\rightarrow I$ the \emph{open dynamical system} generated by the (closed) dynamical system $f:I\rightarrow I$ over $H$.
\end{definition}

The partition $\xi$ will remain a Markov partition of the open dynamical system $f_H:I\rightarrow I$ but the dynamics of the original system $f:I\rightarrow I$ will have been modified such that each point in $H$ is now a fixed point.

For $n\geq 0$ let
\begin{align*}
X^n(f_H)&=\{x\in I:f^n(x)\in H, \ f^k(x)\notin H, \ 0\leq k<n\}\\
      &=\{x\in I:f_H^n(x)\in H, \ f_H^k(x)\notin H, \ 0\leq k<n\}; \ \text{and}\\
Y^n(f_H)&=\{x\in I:f^k(x)\in H, \ \text{for some} \ k, \ 0\leq k\leq n\}\\
      &=\{x\in I:f_H^k(x)\in H, \ \text{for some} \ k, \ 0\leq k\leq n\}.
\end{align*}
The set $X^n(f_H)$ consists of those points that escape through the hole $H$ at time $n$ while $Y^n(f_H)$ are those points that escape through $H$ before time $n+1$.

For the moment suppose $\mu$ is a probability measure on $I$, i.e. $\mu(I)=1$. Then $\mu(X^n(f_H))$ can be treated as the probability that an orbit of $f$ enters $H$ for the first time at time $n$ and $\mu(Y^n(f_H))$ the probability that an orbit of $f$ enters $H$ before time $n+1$. In this regard,
$$P^n(f_H)=1-\mu(Y^n(f_H))$$
represents the probability that a typical point of $I$ does not fall into the hole $H$ by time $n$. For this reason, the quantity $P^n(f_H)$ is called the \emph{survival probability} at time $n$ of the dynamical system $f_H$ for the measure $\mu$.

One of fundamental problems in the theory of open systems is determining or finding ways of approximating $\mu(X^n(f_H))$ and $\mu(Y^n(f_H))$ for finite $n\geq 0$. In the following section we give exact formulae for these quantities in the case where $f_H$ is a piecewise linear function with nonzero slope and $\mu$ is Lebesgue measure. In section \ref{sec:3} we remove the assumption that $f_H$ is piecewise linear and present a method for estimating $\mu(X^n(f_H))$ and $\mu(Y^n(f_H))$ for functions that are piecewise nonlinear.

\section{Piecewise Linear Functions}\label{sec:3}
As a first step, we consider those open systems $f_H:I\rightarrow I$ that are linear when restricted to the elements of the partition $\xi$. More formally, suppose $H=\bigcup_{i\in\mathcal{I}}\xi_i$ for some $\mathcal{I}\subset M$. Let $\mathcal{L}$ be the set of all open systems $f_H:I\rightarrow I$ such that
$$|f_H^{\prime}(x)|=c_i>0 \ \ \text{for} \ \ x\in \xi_i \ \ \text{and} \ \ i\notin\mathcal{I}$$
where each $c_i\in\mathbb{R}$. The set $\mathcal{L}$ consists of all open systems that have a nonzero constant slope when restricted to any $\xi_i\nsubseteq H$.

To each $f_H\in\mathcal{L}$ there is an associated matrix that can be used to compute $\mu(X^n(f_H))$ and $\mu(Y^n(f_H))$. To define this matrix let
\begin{equation}\label{eq:part}
\xi_{ij}=\xi_i\cap f^{-1}(\xi_j) \ \text{for} \ 1\leq i,j\leq m.
\end{equation}

\begin{definition}
Let $f_H\in\mathcal{L}$ where $H=\bigcup_{i\in\mathcal{I}}\xi_i$ for some $\mathcal{I}\subset M$.
The matrix $A_H\in\mathbb{R}^{m\times m}$ given by
$$(A_H)_{ij}=
\begin{cases}
|f^{\prime}(x)|^{-1} \ &\text{for} \ \ x\in\xi_{ij}\neq\emptyset, \ i\notin\mathcal{I},\\
0 &\text{otherwise}
\end{cases} \ \ 1\leq i,j\leq m
$$ is called the \emph{weighted transition matrix} of $f_H$.
\end{definition}

Associated with the open system $f_H:I\rightarrow I$ there is also a directed graph $\Gamma=(V,E_H,\omega)$ with \emph{vertices} $V$ and \emph{edges} $E_H$. For $V=\{v_1,\dots,v_m\}$ we let $e_{ij}$ denote the edge from vertex $v_i$ to $v_j$.

\begin{definition}
 Let $f_H:I\rightarrow I$ where $H=\bigcup_{i\in\mathcal{I}}\xi_i$ for some $\mathcal{I}\subset M$. We define $\Gamma_H=(V,E_H)$ to be the graph with\\
\indent (a) vertices $V=\{v_1,\dots,v_m\}$ and;\\
\indent (b) edges $E_H=\{e_{ij}:cl(\xi_j)\subseteq cl(f(\xi_i)), \ i\notin\mathcal{I}\}$.\\
 The graph $\Gamma_H$ is called the \emph{transition graph} of $f_H$.
\end{definition}

The vertex set $V$ of $\Gamma_H$ represents the elements of the Markov partition $\xi=\{\xi_i\}_{i=1}^m$ and the edge set $E_H$, the possible transitions between the elements of $\xi$. Hence, $e_{ij}\in E_H$ only if there is an $x\in\xi_i\nsubseteq H$ such that $f_H(x)\in\xi_j$, i.e. it is possible to transition from $\xi_i$ to $\xi_j$. We note that as $H=\emptyset$ is a possible hole, the original (closed) system $f:I\rightarrow I$ has a well defined transition graph which we denote by $\Gamma$.

Note that the matrix $A_H$ and the graph $\Gamma_H$ do not carry the same information as $\Gamma_H$ only designates how orbits can transition between elements of $\xi$ whereas $A_H$ additionally gives each of these transitions a weight. However, the graph $\Gamma_H$ gives us a way of visualizing how orbits escape from the system, which will be useful in the following sections. An open system and its transition graph are demonstrated in the following example.

\begin{example}\label{ex:1}
Let the function $f:I\rightarrow I$ be the tent map
$$f(x)=\begin{cases}
2x &  0\leq x\leq 1/2,\\
2-2x & 1/2 < x\leq 1
\end{cases}$$ with Markov partition $\xi=\{(0,1/4],(1/4,1/2],(1/2,3/4],(3/4,1]\}$. Here we consider the hole $H=(0,1/4]$. The open system $f_H:I\rightarrow I$ is shown in figure \ref{fig:1} (left) with the graph $\Gamma_H$ (right).

As $H=\xi_1$ we emphasis this in $\Gamma_H$ by drawing the vertex $v_1$ as an open circle, i.e. as a hole. We note that the only difference between the transition graph $\Gamma$ of $f:I\rightarrow I$ and $\Gamma_H$ is that there are no edges origination from $v_1$ in $\Gamma_H$. In this sense a hole $H$ is an absorbing state since nothing leaves $H$ once it enters.
\end{example}

\begin{figure}
    \begin{tabular}{cc}
    \begin{overpic}[scale=.7]{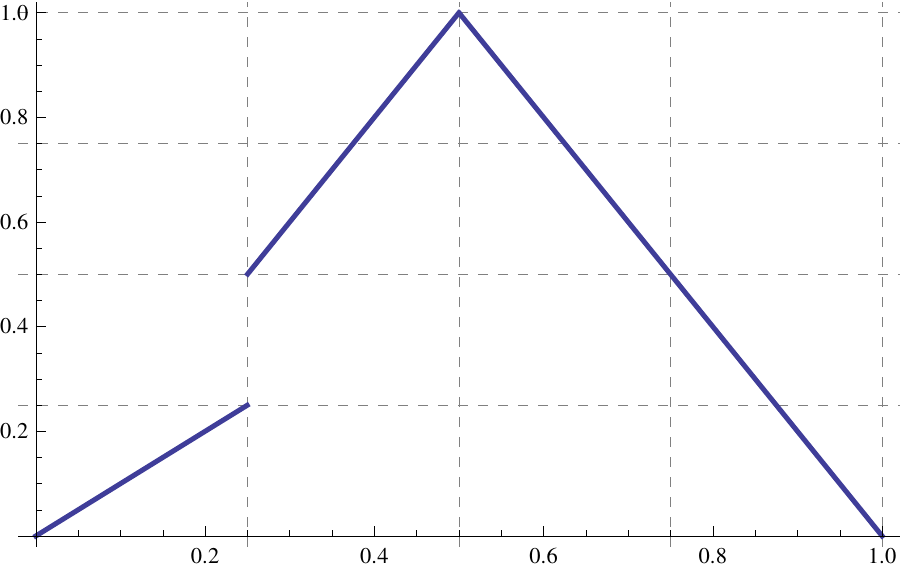}

    \put(37,-5){$\xi_2$}
    \put(6,-5){$\xi_1=H$}
    \put(60,-5){$\xi_3$}
    \put(85,-5){$\xi_4$}

    \end{overpic} &
    \begin{overpic}[scale=.45]{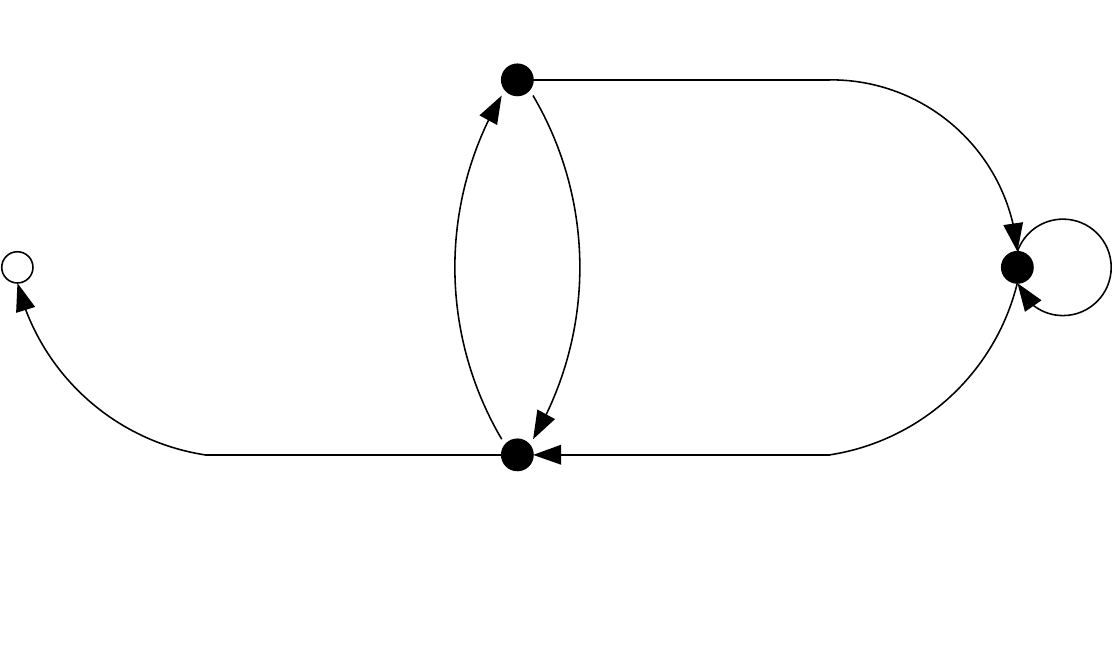}
    \put(42,-1){$\Gamma_H$}




    \put(43,56){$v_2$}
    \put(-2,39){$v_1$}
    \put(43,11){$v_4$}
    \put(91,42){$v_3$}


    \end{overpic}
    \end{tabular}
\caption{The transition graph $\Gamma_H$ (right) of the open system $f_H:I\rightarrow I$ (left) in example \ref{ex:1}.}\label{fig0}\label{fig:1}
\end{figure}
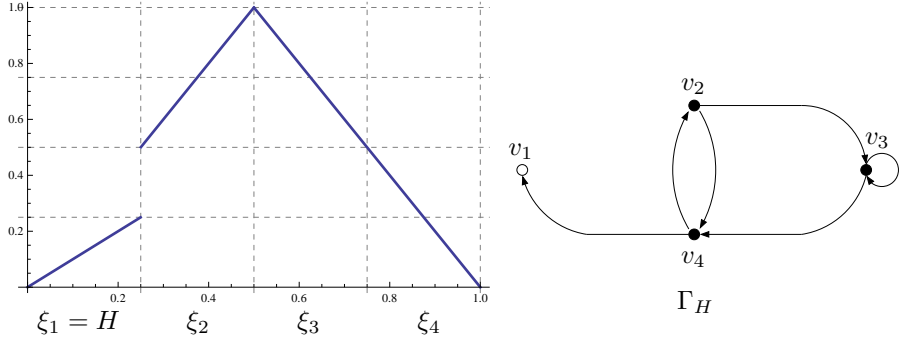

Let $\mathbf{1}=[1,\dots,1]$ be the $1\times m$ vector of ones and $\mathbf{e}_H$ the $m\times 1$ vector given by
$$(\mathbf{e}_H)_{i}=\begin{cases}
\mu(\xi_i) & \text{if} \ \ i\in\mathcal{I}\\
0 & \text{otherwise}
\end{cases}.$$

\begin{theorem}\label{theorem1}
If $f_H\in \mathcal{L}$ and $n\geq 0$ then
\begin{align}
\mu(X^n(f_H))&=\mathbf{1}A_H^n\mathbf{e}_H; \ \text{and}\\
\mu(Y^n(f_H))&=\mathbf{1}\big(\sum_{i=0}^nA_H^i\big)\mathbf{e}_H.
\end{align}
\end{theorem}

For the matrix $B\in \mathbb{R}^{m\times m}$ let $\sigma(B)$ and $\rho(B)$ denote the set of all \textit{eigenvalues} and \textit{spectral radius} of $B$ respectively. We then have the following corollaries to theorem \ref{theorem1}.

\begin{corollary}\label{corollary1}
If $f_H\in \mathcal{L}$ and $\rho(A_H)<1$ then
\begin{align*}
\mu(Y^n(f_H))&=\mathbf{1}(I-A_H)^{-1}(I-A_H^{n+1})\mathbf{e}_H; \ \text{and}\\
\lim_{n\rightarrow\infty}\mu(Y^n(f_H))&=\mathbf{1}(I-A_H)^{-1}\mathbf{e}_H.
\end{align*}
where $I$ is the identity matrix.
\end{corollary}

\begin{corollary}\label{corollary2}
Suppose $f_H\in \mathcal{L}$. If $0<\rho(A_H)<1$ then $\displaystyle{\lim_{n\rightarrow\infty}P^n(f_H)=0}$. If $\rho(A_H)=0$ then $\mu(Y^n(f_H))=1$ for some $n<\infty$.
\end{corollary}

A matrix $B\in\mathbb{R}^{m\times m}$ is called \textit{defective} if it does not have an eigenbasis, i.e. if there are not enough linearly independent eigenvectors of $B$ to form a basis of $\mathbb{R}^m$. A matrix with an eigenbasis is called \textit{nondefective}.

\begin{corollary}\label{corollary3}
Let $f_H\in \mathcal{L}$ and suppose the matrix $A_H$ is nondefective with eigenpairs $\{(\lambda_1,\mathbf{v}_1),\dots,(\lambda_k,\mathbf{v}_k)\}$ with no eigenvalue equal to 1. Then $\mathbf{e}_H=\sum_{i=1}^k c_i \mathbf{v}_i$ for some $c_1,\dots,c_k\in\mathbb{C}$ and
\begin{align}
\mu(X^n(f_H))& =\sum_{i=1}^k c_i s_i\lambda_{i}^{n}\label{eq:3}\\
\mu(Y^n(f_H))& =\sum_{i=1}^k c_i s_i\left(\frac{1-\lambda_{i}^{n+1}}{1-\lambda_i}\right)\label{eq:4}
\end{align}
where $s_i=\mathbf{1}\mathbf{v}_i$.
\end{corollary}

\begin{example}\label{ex:2}
Let the function $f:I\rightarrow I$ be the tent map considered in example \ref{ex:1} and let H=(0,1/4]. As $f_H\in\mathcal{L}$ one can calculate that $f_H$ has the weighted transition matrix
$$
A_H=\left[\begin{array}{cccc}
0 & 0 & 0 & 0\\
0 & 0 & 1/2 & 1/2\\
0 & 0 & 1/2 & 1/2\\
1/2 & 1/2 & 0 & 0
\end{array}\right].
$$

The matrix $A_H$ is nondefective as its eigenvalues $\sigma(A_H)=\{\frac{1+\sqrt{5}}{4},\frac{1-\sqrt{5}}{4},0,0\}$ corresponding respectively to the linearly independent eigenvectors
$$\mathbf{v}_1=\left[
\begin{array}{c}
0\\
\frac{1+\sqrt{5}}{4}\\
\frac{1+\sqrt{5}}{4}\\
1
\end{array}
\right],\mathbf{v}_2=\left[
\begin{array}{c}
0\\
\frac{1-\sqrt{5}}{4}\\
\frac{1-\sqrt{5}}{4}\\
1
\end{array}
\right],
\mathbf{v}_3=\left[
\begin{array}{c}
0\\
0\\
-1\\
1
\end{array}
\right],
\mathbf{v}_4=\left[
\begin{array}{c}
-1\\
1\\
0\\
0
\end{array}
\right].$$
Since the vector $\mathbf{e}_H=[1/4,0,0,0]^T$ can be written as
$$\mathbf{e}_H=\frac{5-\sqrt{5}}{20+20\sqrt{5}}\mathbf{v}_1-\frac{3+\sqrt{5}}{8\sqrt{5}}\mathbf{v}_2+\frac{1}{4}\mathbf{v}_3-\frac{1}{4}\mathbf{v}_4,$$
equations (\ref{eq:3}) and (\ref{eq:4}) in corollary \ref{corollary3} imply
\begin{align*}
\mu(X^n(f_H))&=\frac{1}{40}(5+\sqrt{5})\lambda_1^n+\frac{1}{40}(5-\sqrt{5})\lambda_2^n; \ \text{and}\\
\mu(Y^n(f_H))&=\frac{1}{40}(5+\sqrt{5})\frac{1-\lambda_1^{n+1}}{1-\lambda_1}+\frac{1}{40}(5-\sqrt{5})\frac{1-\lambda_2^{n+1}}{1-\lambda_2}\\
&=1-\left(\frac{1}{2}+\frac{1}{\sqrt{5}}\right)\lambda_1^{n+1}-\left(\frac{1}{2}-\frac{1}{\sqrt{5}}\right)\lambda_2^{n+1}.
\end{align*}

Note that as $\rho(A_H)<1$ then $\lim_{n\rightarrow\infty}P^n(f_H)=0$. Hence, the probability of surviving indefinitely in this system for a typical $x\in I$ is in fact zero. This can be seen in figure \ref{fig:2} where both $\mu(X^n(f_H))$ and $\mu(Y^n(f_H))$ are plotted.
\end{example}

\begin{figure}
    \begin{overpic}[scale=.5]{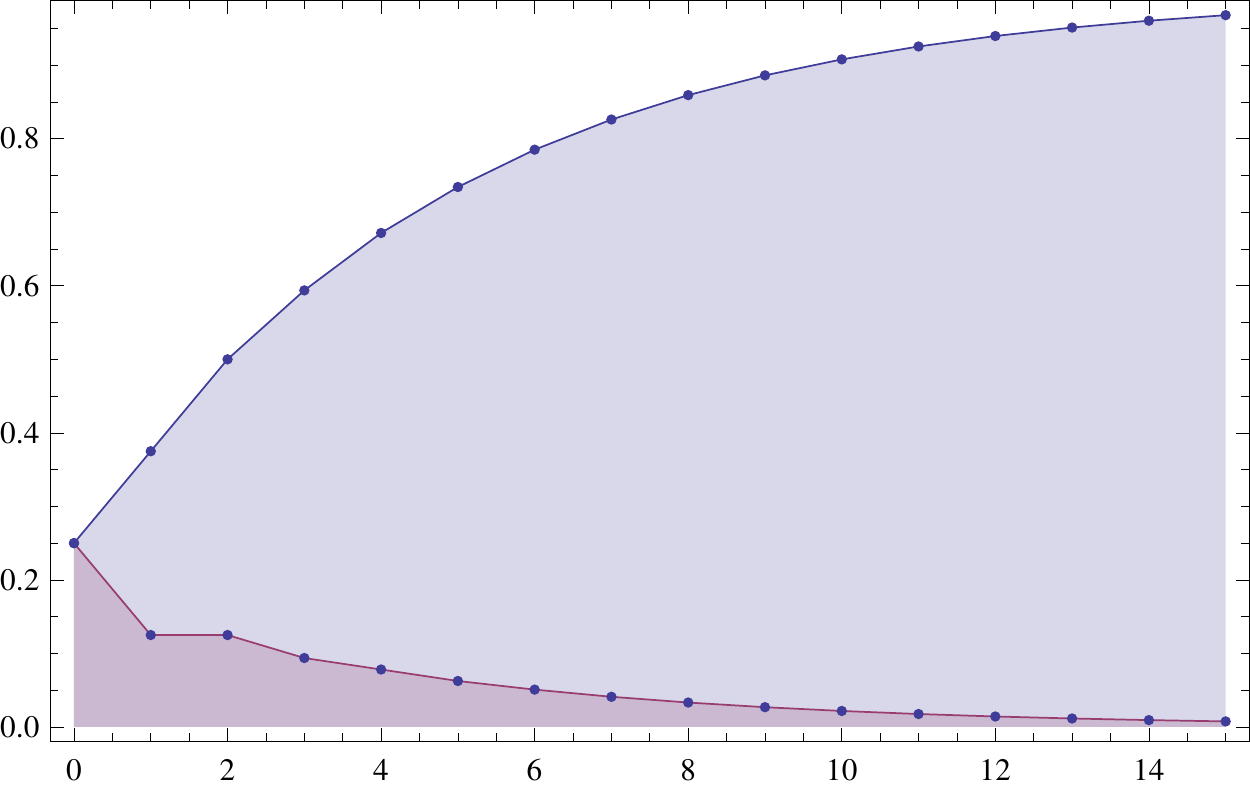}
    \put(40,-5){\small $n-$axis}
    \put(42,12){$\mu(X^n(f_H))$}
    \put(42,46){$\mu(Y^n(f_H))$}
    \end{overpic}
\caption{Plots of $\mu(X^n(f_H))$ and $\mu(Y^n(f_H))$ for $f_H:I\rightarrow I$ in example \ref{ex:2}.}\label{fig:2}
\end{figure}

\section{Nonlinear Estimates}\label{sec:4}
We now consider the open systems $f_H:I\rightarrow I$ where $f$ is allowed to be a nonlinear but differentiable function when restricted to the elements of $\xi$. The formulae we derive in this section allow us to give upper and lower bounds on $\mu(X^n(f_H))$ and $\mu(Y^n(f_H))$ for any finite time $n\geq 0$.

Suppose $H=\bigcup_{i\in\mathcal{I}}\xi_i$ for some $\mathcal{I}\subset M$. Let $\mathcal{N}$ be the open systems $f_H:I\rightarrow I$ such that
\begin{align*}
&\inf_{x\in\xi_i}|f_H^{\prime}(x)|>0 \ \ \text{for} \ \ i\notin\mathcal{I}; \ \ \text{and}\\
&\sup_{x\in\xi_i}|f_H^{\prime}(x)|<\infty \ \ \text{for} \ \ i\notin\mathcal{I}.
\end{align*}
To each $f_H\in\mathcal{N}$ there are two associated matrices similar to the weighted transition matrix $A_H$ defined for each open system in $\mathcal{L}$.

\begin{definition}
Let $f_H\in\mathcal{N}$ where $H=\bigcup_{i\in\mathcal{I}}\xi_i$ for some $\mathcal{I}\subset M$. The matrix $\underline{A}_H\in\mathbb{R}^{m\times m}$ is defined by
$$(\underline{A}_H)_{ij}=\begin{cases}
\displaystyle{\inf_{x\in\xi_{ij}}|f^\prime(x)|^{-1}} & \text{for} \ \ \xi_{ij}\neq\emptyset, \ i\notin\mathcal{I},\\
\ \ \ 0 & \text{otherwise}
\end{cases} \ \ 1\leq i,j\leq m.$$
Similarly, define the matrix $\overline{A}_H\in\mathbb{R}^{m\times m}$ by
$$(\overline{A}_H)_{ij}=\begin{cases}
\displaystyle{\sup_{x\in\xi_{ij}}|f^\prime(x)|^{-1}} & \text{for} \ \ \xi_{ij}\neq\emptyset, \ i\notin\mathcal{I}\\
\ \ \ 0 & \text{otherwise}
\end{cases}  \ \ 1\leq i,j\leq m.$$
\end{definition}

For $f_H\in\mathcal{N}$ and $n\geq 0$ let $$\underline{X}^n(f_H)=\mathbf{1}\underline{A}^n_H\mathbf{e}_H \ \text{and} \ \overline{X}^n(f_H)=\mathbf{1}\overline{A}^n_H\mathbf{e}_H;$$ $$\underline{Y}^n(f_H)=\mathbf{1}\big(\sum_{i=0}^m\underline{A}^i_H\big)\mathbf{e}_H \ \text{and} \ \overline{Y}^n(f_H)=\mathbf{1}\big(\sum_{i=0}^m\overline{A}^i_H\big)\mathbf{e}_H.$$

\begin{theorem}\label{theorem3}
If $f_H\in\mathcal{N}$ and $n\geq 0$ then $\underline{X}^n(f_H) \leq \mu(X^n(f_H)) \leq \overline{X}^n(f_H)$ and $\underline{Y}^n(f_H) \leq \mu(Y^n(f_H)) \leq \overline{Y}^n(f_H).$
\end{theorem}

Theorem \ref{theorem3} allows us to bound the amount of phase space that escapes through $H$ at time $n$ and before time $n+1$. If the matrices $\underline{A}_H$ and $\overline{A}_H$ are nondefective then we have the following result similar to corollary \ref{corollary3}.

\begin{corollary}\label{corollary4}
Let $f_H\in\mathcal{N}$ and suppose both $\underline{A}_H$ and $\overline{A}_H$ are nondefective with eigenpairs $\{(\underline{\lambda}_1,\underline{\mathbf{v}}_1),\dots,(\underline{\lambda}_k,\underline{\mathbf{v}}_k)\}$ and $\{(\overline{\lambda}_1,\overline{\mathbf{v}}_1),\dots,(\overline{\lambda}_k,\overline{\mathbf{v}}_k)\}$ with no eigenvalue equal to 1. Then for each $n\geq 0$
\begin{align}
\sum_{i=1}^k \underline{c}_i \underline{s}_i \underline{\lambda}_{i}^{n}
\leq \mu(X^n(f_H)) \leq &
\sum_{i=1}^k \overline{c}_i \overline{s}_i \overline{\lambda}_{i}^{n}; \ \text{and}\label{eq:7.1}\\
\sum_{i=1}^k \underline{c}_i \underline{s}_i\left(\frac{1-\underline{\lambda}_{i}^{n+1}}{1-\underline{\lambda}_i}\right) \leq \mu(Y^n(f_H)) \leq &
\sum_{i=1}^k \overline{c}_i \overline{s}_i\left(\frac{1-\overline{\lambda}_{i}^{n+1}}{1-\overline{\lambda}_i}\right)\label{eq:8}
\end{align}
where $\underline{s}_i=\mathbf{1}\underline{\mathbf{v}}_i$, $\overline{s}_i=\mathbf{1}\overline{\mathbf{v}}_i$, $\displaystyle{\mathbf{e}_H=\sum_{i=1}^n\underline{c}_i\underline{\mathbf{v}}_i}$ and $\displaystyle{\mathbf{e}_H=\sum_{i=1}^n\overline{c}_i\overline{\mathbf{v}}_i}$.
\end{corollary}

The upper and lower bounds given in (\ref{eq:7.1}) are $\underline{X}^n(f_H)$ and $\overline{X}^n(f_H)$ respectively. The upper and lower bounds given in (\ref{eq:8}) are $\underline{Y}^n(f_H)$ and $\overline{Y}^n(f_H)$ respectively.

\begin{figure}
    \begin{tabular}{cc}
    \begin{overpic}[scale=.7]{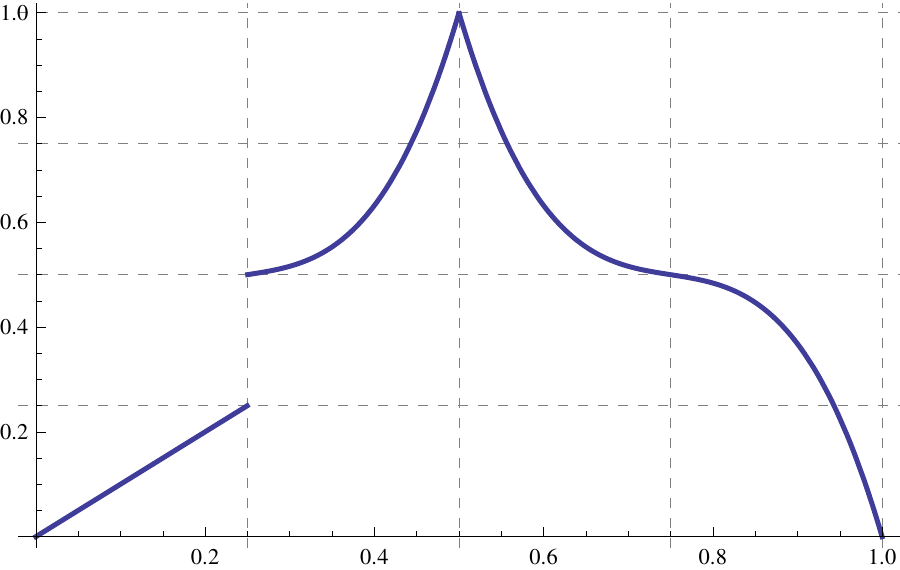}

    \put(37,-5){$\xi_2$}
    \put(6,-5){$\xi_1=H$}
    \put(60,-5){$\xi_3$}
    \put(85,-5){$\xi_4$}

    \end{overpic} &
    \begin{overpic}[scale=.45]{holes1.pdf}
    \put(42,-1){$\Gamma_H$}




    \put(43,56){$v_2$}
    \put(-2,39){$v_1$}
    \put(43,11){$v_4$}
    \put(91,42){$v_3$}


    \end{overpic}
    \end{tabular}
\caption{The transition graph $\Gamma_H$ (right) of the open system $g_H:I\rightarrow I$ (left) in example \ref{ex:3}.}\label{fig:0}
\end{figure}
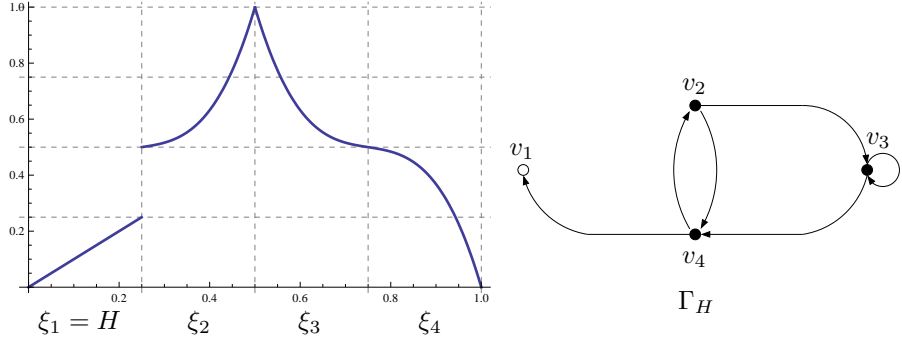

\begin{example}\label{ex:3}
Consider the function $g:I\rightarrow I$ given by
$$g(x)=\begin{cases}
\frac{11}{2}x - 21 x^2 + 28 x^3 &  0\leq x\leq 1/2\\
\frac{11}{2}(1 - x) - 21(1 - x)^2 + 28(1 - x)^3 & 1/2 < x\leq 1
\end{cases}$$ with Markov partition $\xi=\{(0,1/4],(1/4,1/2],(1/2,3/4],(3/4,1]\}$ and $H=(0,1/4]$. The function $g:I\rightarrow I$ can be considered to be a nonlinear version of the tent map $f:I\rightarrow I$ in example \ref{ex:1}. In fact both systems have the same transition graph (see figure \ref{fig:0}).

For $g_H:I\rightarrow I$ one can compute $\xi_{23}=(1/4,.44]$, $\xi_{23}=(.44,1/2]$, $\xi_{23}=(1/2,.55]$, $\xi_{23}=(.55,.3/4]$, $\xi_{23}=(3/4,.94]$, and $\xi_{23}=(.94,1]$. From this we find
$$
\underline{A}_H=\left[\begin{array}{cccc}
0 & 0 & 0 & 0\\
0 & 0 & 0.29 & 2/11\\
0 & 0 & 0.29 & 2/11\\
2/11 & 0.29 & 0 & 0
  \end{array}\right] \ \ \text{and} \ \
\overline{A}_{H}=\left[\begin{array}{cccc}
0 & 0 & 0 & 0\\
0 & 0 & 4 & 0.29\\
0 & 0 & 4 & 0.29\\
0.29 & 4 & 0 & 0
  \end{array}\right].
$$

As $\underline{A}_H$ and $\overline{A}_H$ are nondefective then using corollary \ref{corollary4} one can compute that
\begin{equation}\label{eq:7}
0.15\underline{\lambda}_1^n+0.15\underline{\lambda}_2^n-0.6\leq X^n(g_H)\leq -0.02\overline{\lambda}_1^n+0.02\overline{\lambda}_2^n+.25
\end{equation}
where $\underline{\lambda}_1=0.41$, $\underline{\lambda}_2=-0.12$, $\overline{\lambda}_1=4.27$, and $\overline{\lambda}_2=-0.27$. Plotting the inequalities in (\ref{eq:7}) yields the picture in figure \ref{fig:4}. Here the shaded area indicates the region in which $\mu(X^n(g_H))$ must lie.
\end{example}

\begin{figure}
    \begin{overpic}[scale=.65]{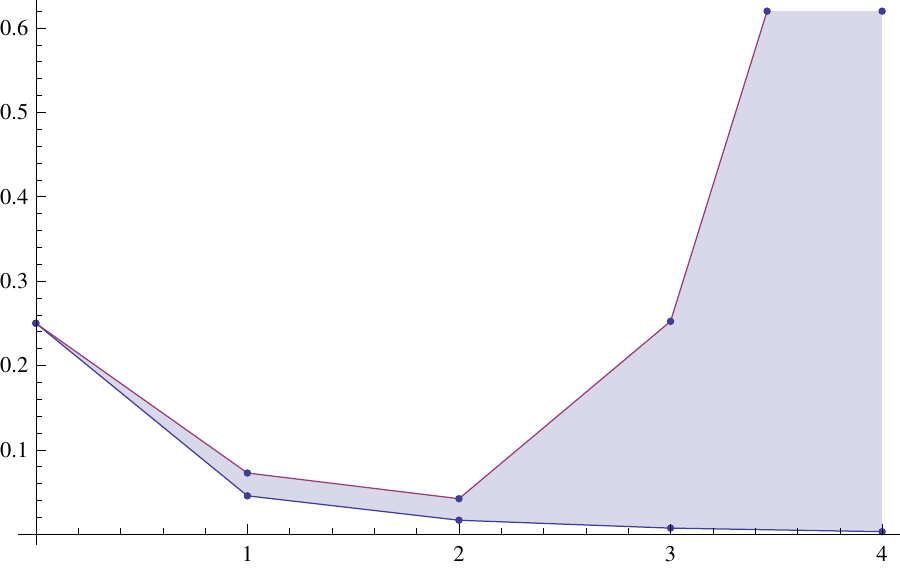}
    \put(65,8){$\underline{X}^n(g_H)$}
    \put(52,28){$\overline{X}^n(g_H)$}
    \put(45,-2.5){\small $n-$axis}
    \end{overpic}
\caption{The upper bounds $\overline{X}^n(g_H)$ and lower bounds $\underline{X}^n(g_H)$ for $\mu(X^n(g_H))$ are shown for the open system $g_H:I\rightarrow I$ in example \ref{ex:3}.}\label{fig:4}
\end{figure}

\section{Improved Escape Estimates}
In this section we define a delayed first return map of an open system $f_H\in\mathcal{N}$, which we will use to improve the escape estimates given in theorem \ref{theorem3}. A key step in this procedure is to choose a particular vertex set of $\Gamma_H$ over which this map will be defined. This requires that we know the cycle structure of $\Gamma_H$.

A \textit{path} $\mathcal{P}$ in the graph $\Gamma_H=(V,E_H)$ is an ordered sequence of distinct vertices $v_1,\dots,v_k\in V$ such that $e_{i,i+1}\in E$ for $1\leq i\leq k-1$. If the vertices $v_1$ and $v_k$ are the same then $\mathcal{P}$ is a \textit{cycle}. If $S\subseteq V$ where $V$ is the vertex set of the graph we will write $\bar{S}=V\setminus S$.

\begin{definition}
Let $H=\bigcup_{i\in\mathcal{I}}\xi_i$ for some $\mathcal{I}\subset M$ and let $\Gamma_H=(V,E_H)$. The set $S\subseteq V=\{v_1,\dots,v_m\}$ is an \emph{open structural set} of $\Gamma_H$ if $v_i\in S$ for $i\in\mathcal{I}$ and $\Gamma_H|_{\bar{S}}$ has no cycles.
\end{definition}

Structural sets were first defined in \cite{BW10} where they were used to gain improved estimates of a dynamical network's stability. Later in \cite{BW09} they were used to improve the eigenvalue estimates of Gershgorin et. al. Here, our goal is to extend their use to improve our estimates of $\mu(X^n(f_H))$ and $\mu(Y^n(f_H))$.

For the open system $f_H:I\rightarrow I$ we let $st(\Gamma_H)$ denote the set of all open structural sets of $\Gamma_H$. If $S\in st(\Gamma_H)$ we let $\mathcal{I}_S=\{i\in M: v_i\in S\}$ be the \emph{index set} of $S$ and $\xi_S=\bigcup_{i\in\mathcal{I}_S}\xi_i$.

\begin{definition}
Let $S\in st(\Gamma_H)$. For $x\in I$ we let $\gamma(x)=i_0i_1\dots i_t$ where $i_j=k$ if $f^j_H(x)\in\xi_k$ and $t$ is the smallest $k>0$ such that $f^k_H(x)\in\xi_S$. The set
$$\Omega_S=\{\gamma:\gamma=\gamma(x) \ \text{for some} \ x\in I\setminus H\}$$
are the \emph{admissible sequences} of $f_H$ with respect to $S$.
\end{definition}

For $x\in I$ we say $\gamma(x)=i_0i_1\dots i_t$ has length $|\gamma(x)|=t$. The reason $|\gamma(x)|<\infty$ is that the graph $\Gamma_H|_{\bar{S}}$ has no cycles. Hence, after a finite number of steps $f^t_H(x)$ must enter $\xi_S$.

\begin{definition}\label{def:dfrmap}
Let $S\in st(\Gamma_H)$. For $x_0\in I$ and $k\geq 0$ we inductively define $x_{k+1}=Rf_S(x_k,\dots,x_0)$ where
$$x_{k+1}=
\begin{cases}
f_H^{|\gamma(x_k)|}(x_k) &\text{if} \ \ x_{k-i}=x_k \ \ \text{for each} \ \ 0\leq i\leq |\gamma(x_k)|-1,\\
x_k \ \ &\text{otherwise}.
\end{cases}
$$ The function $Rf_S:I^{k+1}\rightarrow I$ is called the \emph{delayed first return map} of $f_H$ with respect to $S$. The sequence $x_0,x_1,x_2,\dots$ is the \emph{orbit} of $x_0$ under $Rf_S$.
\end{definition}

If $T=\max_{x\in I}|\gamma(x)|$ then strictly speaking $x_{k+1}=Rf_S(x_k,\dots,x_{\tau})$ for some $\tau<T$. The map $Rf_S$ acts almost like a first return map of $f_H$ to the set $\xi_S$. The difference is that a return to $\xi_S$ does not happen instantaneously (as it would happen in the case of a first return map) but is delayed so that the trajectory of a point under $f_H$ and $Rf_S$ coincide after a return to $\xi_S$.

For $n\geq 0$ we let $Rf_S^n(x_0)=x_n$ and define
\begin{align*}
X^n(Rf_S)&=\{x\in I:Rf_S^n(x)\in H, \ Rf_S^k(x)\notin H, \ 0\leq k<n\}; \ \text{and}\\
Y^n(Rf_S)&=\{x\in I:Rf_S^k(x)\in H, \ \text{for some} \ k, \ 0\leq k\leq n\}.
\end{align*}

\begin{lemma}\label{lemma:1}
If $S\in st_{\xi}(\Gamma_{f_H})$ and $n\geq 0$ then $X^n(f_H)=X^n(Rf_S)$.
\end{lemma}

\begin{proof}
For $x_0\in I$ let $\tilde{\gamma}(x_0)=i_0i_1\dots$ where $i_j=k$ if $f^j_{H}(x_0)\in\xi_k$. Choosing $S\in st(\Gamma_H)$ let $\tilde{\gamma}_S(x_0)=\ell_0\ell_1\dots$ where $\ell_j=k$ if $Rf_S^j(x_0)\in\xi_k$. Let $t>0$ be the smallest number such that $i_t\in\mathcal{I}_S$. Then $\gamma(x_0)=i_0i_1\dots i_t$ and definition \ref{def:dfrmap} implies $Rf_S^t(x_0)=f^t_H(x_0)$. Therefore, $i_t=\ell_t$ where $t\in\mathcal{I_S}$.

Continuing in this manner it follows that $i_j=\ell_j$ for each $j\in\mathcal{I}_S$. Since $H\subseteq\xi_S$ the point $x_0$, if it escapes, will escape for both $f_H$ and $Rf_S$ at exactly the same time. This implies the result.
\end{proof}

The major idea in this section is that one can use $Rf_S$ to study the escape of $f_H$ through $H$. However, the weighted transition matrix of $Rf_S:I^k\rightarrow I$ cannot be defined in the same way that we have defined either $\underline{A}_H$ or $\overline{A}_H$. To define a transition matrix for $Rf_S$ we require the following.

For $S\in st(\Gamma_H)$ let
$$M_S=M\cup\{\gamma;i:\gamma\in\Omega_S,0<i<|\gamma|\}$$
If $\gamma=i_0\dots,i_t$ we identify the index $\gamma;0$ with $i_0$ and the index $\gamma;t$ with $i_t$. We also let $\xi_\gamma=\{x\in I: \gamma(x)=\gamma\}$ for each admissible sequence $\gamma\in\Omega_S$, which simply extends our notation given by (\ref{eq:part}) in section \ref{sec:3}.

\begin{definition}\label{def:rutmat}
For $S\in st(\Gamma_{H})$ let $\underline{A}_S$ be the matrix with rows and columns indexed by elements of $M_S$ where
\begin{equation}\label{def:rutmat}
(\underline{A}_S)_{ij}=
\begin{cases}
\displaystyle{\inf_{x\in\xi_\gamma}|\big(f^{|\gamma|}(x)\big)^{\prime}|^{-1}} &\text{if} \ i=\gamma;|\gamma|-1, \ j=\gamma;|\gamma|, \ \text{for some} \ \gamma\in\Omega_S\\
1 &\text{if} \ i=\gamma;k-1, \ j=\gamma;k, \ k\neq |\gamma|,  \ \text{for some} \ \gamma\in\Omega_S\\
0 &\text{otherwise}.
\end{cases}
\end{equation}
We call $\underline{A}_S$ the \emph{lower transition matrix} of $Rf_S$. The matrix $\overline{A}_S$ defined by replacing the infimum in (\ref{def:rutmat}) by a supremum is the \emph{upper transition matrix} of $Rf_S$.
\end{definition}

Let $\mathbf{1}_S$ be the $1\times|M_S|$ vector given by
$$(\textbf{1}_S)_{i}=\begin{cases} 1 & \text{if} \ \ i\in M,\\
0 & \text{otherwise}.
\end{cases}$$
Let $\mathbf{e}_S$ be the $|M_S|\times 1$ vector given by
$$(\textbf{e}_S)_{i}=\begin{cases} \mu(\xi_i) & \text{if} \ \ i\in \mathcal{I},\\
0 & \text{otherwise}.
\end{cases}$$
Lastly, for $n\geq 0$ let
$$\underline{X}^n(Rf_S)=\mathbf{1}_S\underline{A}^n_S\mathbf{e}_S \ \text{and} \ \overline{X}^n(Rf_S)=\mathbf{1}_S\overline{A}^n_S\mathbf{e}_S;$$ $$\underline{Y}^n(Rf_S)=\mathbf{1}_S\big(\sum_{i=0}^n\underline{A}^i_S\big)\mathbf{e}_S \ \text{and} \ \overline{Y}^n(Rf_S)=\mathbf{1}_S\big(\sum_{i=0}^n\overline{A}^i_S\big)\mathbf{e}_S.$$
Using these quantities we give the following improved escape estimates.

\begin{theorem}\label{thm:main}
Let $f_H\in\mathcal{N}$ and suppose $S\in st(\Gamma_H)$. If $n\geq 0$ then
\begin{align*}
\underline{X}^n(f_H) &\leq \underline{X}^n(Rf_S) \leq \mu(X^n(f_H)) \leq \overline{X}^n(Rf_S) \leq \overline{X}^n(f_H); \ \text{and}\\
\underline{Y}^n(f_H) &\leq \underline{Y}^n(Rf_S) \leq \mu(Y^n(f_H)) \leq \overline{Y}^n(Rf_S) \leq \overline{Y}^n(f_H).
\end{align*}
\end{theorem}

Theorem \ref{thm:main} together will lemma \ref{lemma:1} imply that the escape of $f_H$ through $H$ is better approximated by considering any of its delayed first return maps $Rf_S$ than $f_H$ itself. We now give a proof of theorem \ref{thm:main}.

\begin{proof}
For $S\in st(\Gamma_H)$ suppose $i\in M\setminus\mathcal{I}$ and $j\in\mathcal{I}$. Then
$$(\underline{A}_S)_{ij}(\mathbf{e}_S)_j=
\begin{cases}
\inf_{x\in\xi_{ij}}|f^\prime(x)|^{-1}\mu(\xi_j) \ &\text{if} \ \xi_{ij}\neq\emptyset,\\
0 &\text{otherwise}
\end{cases} \
\leq \ \mu\{x\in\xi_i:f_H(x)\in\xi_j\}.$$

To show that a similar formula holds for larger powers of $\underline{A}_S$ suppose $k\in\mathcal{I}_S$. If $ik,kj\in\Omega_S$ then
\begin{align*}
(\underline{A}_S)_{ik}(\underline{A}_S)_{kj}(\mathbf{e}_S)_j
=&\inf_{x\in\xi_{ik}}|f^\prime(x)|^{-1} \inf_{x\in\xi_{kj}}|f^\prime(x)|^{-1}\mu(\xi_j)\\
\leq&\mu\{x\in\xi_i:f_H(x)\in\xi_k,f^2_H(x)\in\xi_j\}.
\end{align*}
If either $ik\notin\Omega_S$ or $kj\notin\Omega_S$ then $(\underline{A}_S)_{ik}(\underline{A}_S)_{kj}(\mathbf{e}_S)_j=0$.

Suppose $k\in M\setminus\mathcal{I}_S$. If $ikj\in\Omega_S$ then $ikj;1\in M_S$ and
\begin{align*}
(\underline{A}_S)_{i,ikj;1}(\underline{A}_S)_{ikj;1,j}(\mathbf{e}_S)_j
=&1\cdot\inf_{x\in\xi_{ikj}}|(f^2(x))^\prime|^{-1}\mu(\xi_j)\\
\leq&\mu\{x\in\xi_i:f_H(x)\in\xi_k,f^2_H(x)\in\xi_j\}.
\end{align*}
If $ikj\notin\Omega_S$ then $ijk;1\notin M_S$. Since
\begin{align*}
(\underline{A}_S^2)_{ij}(\mathbf{e}_S)_j&=\sum_{k\in M_S}(\underline{A}_S)_{ik}(\underline{A}_S)_{kj}\mu(\xi_j)\\
&=\sum_{ik,kj\in \Omega_S}(\underline{A}_S)_{ik}(\underline{A}_S)_{kj}\mu(\xi_j)+\sum_{ikj\in \Omega_S}(\underline{A}_S)_{i,ikj;1}(\underline{A}_S)_{ikj;1,j}\mu(\xi_j)\\
&\leq\sum_{k\in\mathcal{I}_S\cup (M\setminus\mathcal{I}_S)}\mu\{x\in\xi_i:f_H(x)\in\xi_k,f^2_H(x)\in\xi_j\}\\
&=\mu\{x\in\xi_i:f^2_H(x)\in\xi_j\}.
\end{align*}

Continuing in this manner it follows that
\begin{equation}\label{eq:10}
(\underline{A}_S^n)_{ij}(\mathbf{e}_S)_j\leq\mu\{x\in\xi_i:f^n_H(x)\in\xi_j\}
\end{equation}
for $i\in M\setminus\mathcal{I}$, $j\in\mathcal{I}$, and $n\geq 1$. Since $(\mathbf{e}_S)_j=0$ if $j\notin \mathcal{I}$ then for $n\geq 1$ equation (\ref{eq:10}) implies
\begin{align*}
\mathbf{1}_S\underline{A}_S^n\mathbf{e}_S=\sum_{i\in M}\sum_{j\in M_S}(\underline{A}_S^n)_{ij}(\mathbf{e}_S)_j&\leq\sum_{i\in M\setminus\mathcal{I}}\mu\{x\in\xi_i:f^n_H(x)\in H\}\\
&=\mu\{x\in I\setminus H: f^n_H(x)\in H\}.
\end{align*}
As $\mathbf{1}_S\underline{A}_S^0\mathbf{e}_S=\mu(H)$ then $\underline{X}^n(Rf_S) \leq \mu(X^n(f_H)) \leq \overline{X}^n(Rf_S)$ for $n\geq 0$ where the second inequality follows by using the same argument with the matrix $\overline{A}_S$.

To show that $\underline{X}^n(f_H) \leq \underline{X}^n(Rf_S)$ we again suppose that $i\in M\setminus\mathcal{I}$ and $j\in\mathcal{I}$. Then we have
$$(\underline{A}_H)_{ij}(\mathbf{e}_H)_j=
\begin{cases}
\inf_{x\in\xi_{ij}}|f^\prime(x)|^{-1}\mu(\xi_j) \ &\text{if} \ \xi_{ij}\neq\emptyset,\\
0 &\text{otherwise}
\end{cases} \
= \ (\underline{A}_S)_{ij}(\mathbf{e}_S)_j.$$
For larger matrix powers we have
\begin{align*}
\sum_{k\in M\setminus\mathcal{I}_S}(\underline{A}_H)_{ik}(\underline{A}_H)_{kj}\mu(\xi_j)
=&\sum_{k\in M\setminus\mathcal{I}_S}\inf_{x\in\xi_{ik}}|f^\prime(x)|^{-1} \inf_{x\in\xi_{kj}}|f^\prime(x)|^{-1}\mu(\xi_j)\\
\leq\sum_{k\in M\setminus\mathcal{I}_S}1\cdot\inf_{x\in\xi_{ikj}}|(f^2(x))^\prime|^{-1}\mu(\xi_j)
=&\sum_{ikj\in \Omega_S}(\underline{A}_S)_{i,ikj;1}(\underline{A}_S)_{ikj;1,j}(\mathbf{e}_S)_j.
\end{align*}
From this it follows that
\begin{align*}
(\underline{A}_H^2)_{ij}(\mathbf{e}_H)_j
=& \sum_{k\in \mathcal{I}_S}(\underline{A}_H)_{ik}(\underline{A}_H)_{kj}\mu(\xi_j) +  \sum_{k\in M\setminus\mathcal{I}_S}(\underline{A}_H)_{ik}(\underline{A}_H)_{kj}\mu(\xi_j)\\
\leq& \sum_{ik,kj\in\Omega_S}(\underline{A}_S)_{ik}(\underline{A}_S)_{kj}\mu(\xi_j) +  \sum_{ikj\in\Omega_S}(\underline{A}_H)_{i,ikj;1}(\underline{A}_H)_{ikj;1,j}\mu(\xi_j)\\
=&(\underline{A}_S^2)_{ij}(\mathbf{e}_S)_j.
\end{align*}
Again, continuing in this manner we have
$(\underline{A}_H^n)_{ij}(\mathbf{e}_H)_j\leq (\underline{A}_S^n)_{ij}(\mathbf{e}_S)_j$ for $i\in M\setminus\mathcal{I}$, $j\in\mathcal{I}$, and $n\geq 1$. As $\mathbf{1}_H\underline{A}_H^0\mathbf{e}_H=\mu{H}=\mathbf{1}_S\underline{A}_S^0\mathbf{e}_S$ then
\begin{equation*}
\mathbf{1}_H\underline{A}_H^n\mathbf{e}_H=\sum_{i\in M}\sum_{j\in \mathcal{I}}(\underline{A}_H^n)_{ij}(\mathbf{e}_H)_j\leq\sum_{i\in M}\sum_{j\in M_S}(\underline{A}_S^n)_{ij}(\mathbf{e}_S)_j=\mathbf{1}_S\underline{A}_S^n\mathbf{e}_S
\end{equation*}
for $n\geq 0$. Hence, $\underline{X}^n(f_H) \leq \underline{X}^n(Rf_S)$.

By using the same argument with the matrix $\overline{A}_S$ we obtain the inequality $\underline{X}^n(Rf_S) \leq \underline{X}^n(f_H)$. The second set of inequalities in theorem \ref{thm:main} then follow, which completes the proof.
\end{proof}

\begin{figure}
    \begin{overpic}[scale=.7]{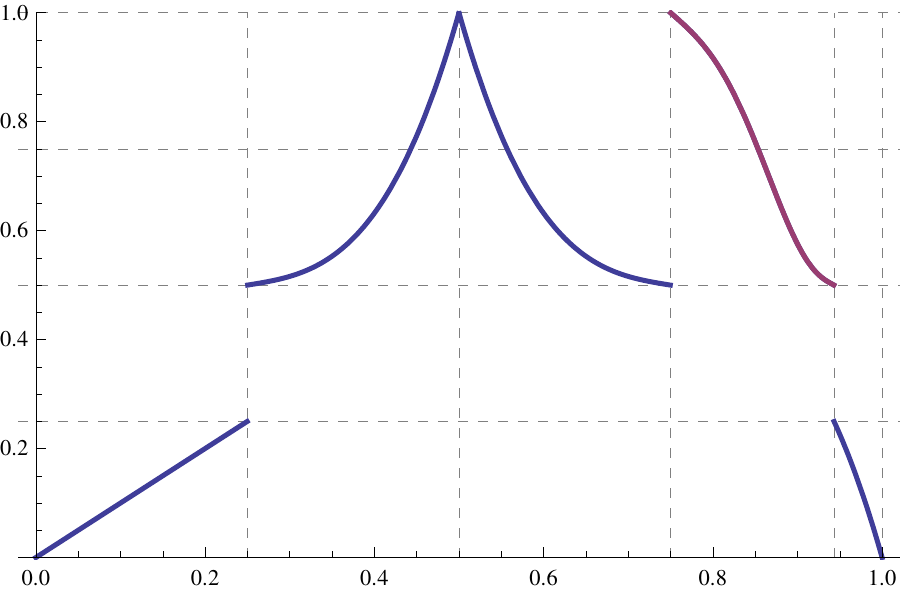}
    \put(37,-5){$\xi_2$}
    \put(6,-5){$\xi_1=H$}
    \put(60,-5){$\xi_3$}
    \put(81,-5){$\xi_{42}$}
    \put(92,-5){$\xi_{41}$}
    \end{overpic}
\caption{The delayed first return map $Rg_S:I\rightarrow I$ in example \ref{ex:4} where $Rg_S$ is delayed on the set $\xi_{42}$ shown in red.}\label{fig:5}
\end{figure}
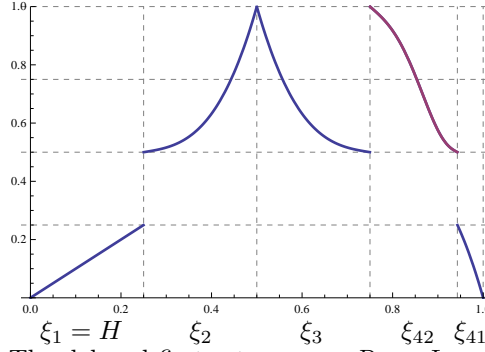

\begin{example}\label{ex:4} Consider the open system $g_H:I\rightarrow I$ given in example \ref{ex:3}. Observe that the vertex set $S=\{v_1,v_3,v_4\}$ is an open structural set of $\Gamma_H$ since $v_1$ is the vertex that corresponds to $H$ and the graph $\Gamma_H|_{\bar{S}}=\{v_2\}$ has no cycles (see figure \ref{fig:0}).

The delayed first return map $Rg_S$ can be written as
$$Rg_S(x_k)=
\begin{cases}
g_H(x_k) \ \ &\text{if} \ x_k\notin \xi_{42},\\
g_H^2(x_k) \ \ &\text{if} \ x_k,x_{k-1}\in \xi_{42},\\
x_k          &\text{otherwise}.
\end{cases}$$
The map $Rg_S$ is shown in figure \ref{fig:5} as a one-dimensional map but is colored red on $\xi_{42}$ to indicate that in fact the system is delayed on this set. Specifically, $g_H(x)$ is shown in blue and $g^2_H(x)$ is shown in red where the trajectories of $Rg_H$ stay in $\xi_{42}$ for two time-steps before leaving.

To compute the upper and lower transition matrices of $Rf_H$ note that the system's admissible sequences are given by $\Omega_S=\{23,24,33,34,424,423,41\}$ implying
\begin{equation}\label{eq:last}
M_S=\{1,2,3,4,424;1,423;1\}.
\end{equation}
From $\Omega_S$ we compute that $\xi_{424}=(3/4,.85]$ and $\xi_{423}=(.85,.94]$. The other partition elements have been computed in example \ref{ex:3}. Using the order given in (\ref{eq:last}) we obtain
$$
\underline{A}_S=\left[\begin{array}{cccccc}
0 & 0 & 0&0&0&0\\
0 & 0 & .29&2/11&0&0\\
0 & 0 & .29&2/11&0&0\\
2/11 & 0 & 0&0&1&1\\
0 & 0 & 0&.26&0&0\\
0 & .26 & 0&0&0&0\\
  \end{array}\right], \ \
\overline{A}_S=\left[\begin{array}{cccccccc}
0 & 0 & 0&0&0&0\\
0 & 0 & 4&.29&0&0\\
0 & 0 & 4&.29&0&0\\
.29 & 0 & 0&0&1&1\\
0 & 0 & 0&.73&0&0\\
0 & 1.18 & 0&0&0&0\\
  \end{array}\right].
$$
The vectors $\mathbf{1}_S$ and $\mathbf{e}_S$ are given by
$$\mathbf{1}_S=[1,1,1,1,0,0] \ \text{and} \ \mathbf{e}_S=[1/4,0,0,0,0,0]^T.$$
Figure \ref{fig:6} shows that $\underline{X}^n(g_H)<\underline{X}^n(Rg_S)$ and $\overline{X}^n(Rg_S)<\overline{X}^n(g_H)$ for a number of $n-$values and indicates the extent to which using the delayed first return map $Rg_H$ improves our estimates of $\mu(X^n(g_H))$. The shaded regions in these graphs represent the difference between these upper and lower estimates respectively.

\end{example}

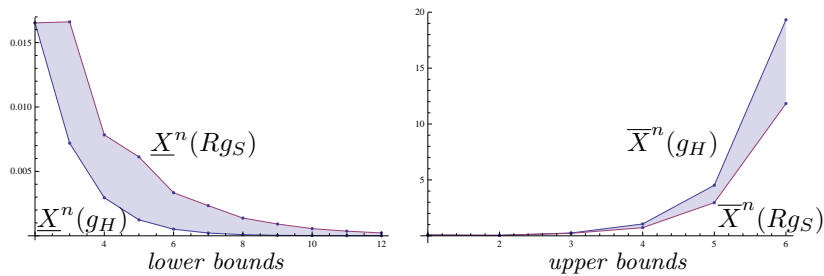
\begin{figure}
    \begin{tabular}{cc}
    \begin{overpic}[scale=.55]{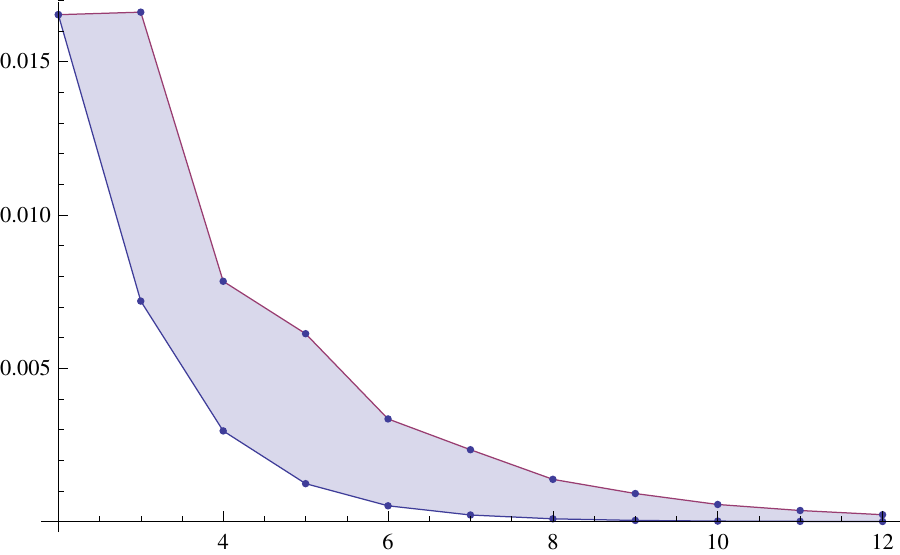}

    \put(37,26){$\underline{X}^n(Rg_S)$}
    \put(7,6){$\underline{X}^n(g_H)$}
    \put(36,-5){\small \emph{lower bounds}}

    \end{overpic} &
    \begin{overpic}[scale=.55]{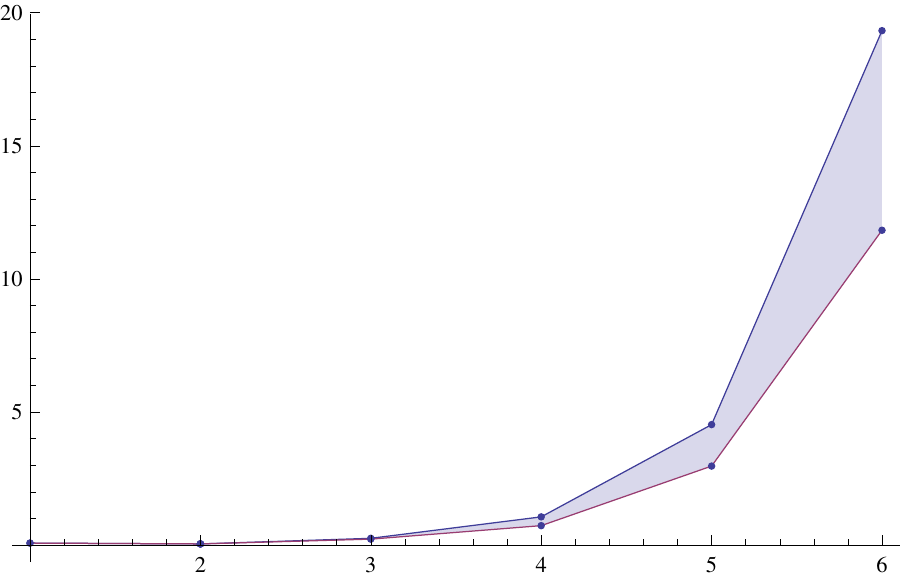}

    \put(56,26){$\overline{X}^n(g_H)$}
    \put(80,6){$\overline{X}^n(Rg_S)$}
    \put(35,-5){\small \emph{upper bounds}}

    \end{overpic}
    \end{tabular}
\caption{Comparison between $\underline{X}^n(g_H)$ and $\underline{X}^n(Rg_H)$ (left) and $\overline{X}^n(g_H)$ and $\overline{X}^n(Rg_H)$ (right) from examples \ref{ex:2} and \ref{ex:3}.}\label{fig:6}
\end{figure}

\section{Conclusion}
Our results demonstrate that the theory of isospectral transformations can be effectively applied to open dynamical systems to obtain sharper estimates of a system's survival probabilities. Previous applications of this theory have allowed for improvements in the classical estimates of matrix spectra and in obtaining stronger sufficient conditions for the global stability of dynamical networks. We do not doubt that this theory can be applied to another problems as well.

Concerning the results of the present paper it is easy to see that the improved estimates we obtained can also be found for a much broader class of open systems via the same technique. For instance, one could extend these techniques to higher dimensional systems.


\begin{thebibliography}{9}
\bibitem{AB10} V. S. Afraimovich, L. A. Bunimovich, Which Hole is Leaking the Most: A Topological Approach to Study Open Systems, \textit{Nonlinearity} \textbf{23}, 2010, 644-657.

\bibitem{AB07} V. Afraimovich and L. Bunimovich. Dynamical networks: interplay of topology, interactions, and local dynamics, \textit{Nonlinearity} \textbf{20}, 2007, 1761-1771.

\bibitem{BB11} Yu. Bakhtin, L. A. Bunimovich, The Optimal Sink and the Best Source in a Markov Chain \textit{J. Stat. Physics} \textbf{143}, 2011, 943-954.

\bibitem{BB05} M. L. Blank, L. A. Bunimovich, Long Range Action in Networks of Chaotic Elements, \textit{Nonlinearity} \textbf{19}, 2006, 330-345.

\bibitem{B12} L. A. Bunimovich, Fair Dice-Like Hyperbolic Systems, \textit{Contemporary Mathematics} \textbf{467}, 2012, 79-87.

\bibitem{BW10} L. A. Bunimovich, B. Z. Webb, Isospectral Graph Transformations, Spectral Equivalence, and Global Stability of Dynamical Networks. \textit{Nonlinearity} \textbf{25}, 2012, 211-254.

\bibitem{BW09} L. A. Bunimovich, B. Z. Webb, Dynamical Networks, Isospectral Graph Reductions, and Improved Estimates of Matrices' Spectra, \textit{Linear Algebra and its Applications} \textbf{437}, Issue 7, 1429-1457, 2012.

\bibitem{BY10} L. A. Bunimovich, A. Yurchenko, Where to Place a Hole to Achieve the Fastest Escape Rate, \textit{Israel J. Math.} \textbf{182}, 2011, 229-252.

\bibitem{DW12} M. Demers, P. Wright, Behavior of the Escape Rate Function in Hyperbolic Dynamical Systems, \textit{Nonlinearity} \textbf{25}, 2012, 2133-2150.

\bibitem{Varga09} R. Varga, Gershgorin and His Circles (Germany: Springer-Verlag Berlin Heidelberg) 2004.

\end{thebibliography}
\end{document}